\documentclass[reqno]{article}
\usepackage{amssymb,amsmath,amsthm}

\usepackage[T1]{fontenc}
\usepackage[latin1]{inputenc}
\usepackage{graphicx}
\usepackage{enumerate,ae}

\newtheorem{Theo}{Theorem}

\newtheorem{Lem}[Theo]{Lemma}
\newtheorem{Prop}[Theo]{Proposition}

\theoremstyle{remark}

\newtheorem{exx}{Example}

\def\E{\mathbb{E}}
\def\Z{\mathbb{Z}}

\def\epsilon{\varepsilon}

\newcommand{\set}[1]{\left\{#1\right\}}

\def\tops{\xrightarrow[]{a.s.}}

\def\d{\ \mathrm{d}}
\def\var{\mathrm{Var}}

\begin{document}

 \title{Random sampling of long-memory stationary processes}

\author {Anne  Philippe$^{\text\small{1}}$  \,\, and \, Marie-Claude Viano$^{\text\small{2}}$ \\
\
{\small $^{^{\text\small{1}}}$
 Universit\'e de Nantes, Laboratoire de Math\'ematiques Jean Leray,
UMR CNRS 6629}  \\
{\small 
2 rue de la Houssini\`ere - BP 92208, 
44322 Nantes Cedex 3, France}\\
{\small $^{\text\small{2}}$ Laboratoire Paul Painlev\'e  UMR CNRS 8524,
 UFR de Math\'ematiques -- Bat M2} \\
{\small 
 Universit\'e de Lille 1, Villeneuve d'Ascq, 59655 Cedex, France}} 
\date{}
\maketitle
\begin{abstract}


  This paper investigates the second order properties of a stationary
  process after random sampling. While a short memory process gives
  always rise to a short memory one, we prove that long-memory can
  disappear when the sampling law has heavy enough tails. We prove that under
  rather general conditions  the existence of the spectral density is
  preserved by random sampling. We also investigate the effects of deterministic
  sampling on seasonal long-memory.
\end{abstract}

\vskip .5cm 
  
{\bf keywords:}   {Aliasing, FARIMA Processes, Generalised Fractional
  Processes, Regularly varying covariance, Seasonal long-memory, Spectral density}

\textbf{MS 2000 Mathematics Subject Classification: } 60G10 60G12 62M10 62M15

\section{Introduction}

The effects of random sampling on the second order characteristics and more
specially on the memory of a stationary second order discrete
time process are the subject of this paper.

We start from $\mathbf{X}=(X_n)_{n\geq 0}$, a stationary discrete time
second order process with covariance sequence $\sigma_X(h)$ and a
random walk $(T_n)_{n\geq 0}$ independent of $\mathbf{X}$.  The
sampling intervals $\Delta_{j}=T_{j}-T_{j-1}$ are independent
identically distributed integer random variables with common
probability law $S$.  We fix $T_0=0$.

Throughout the paper we consider the
sampled process $\mathbf{Y}$ defined by
\begin{equation}\label{def}
Y_n=X_{T_n}\quad n=0,1\ldots
\end{equation}

The particular case where $S=\delta_k$, the Dirac measure at point $k$, shall  be
mentioned as deterministic sampling (some authors prefer systematic or
periodic sampling).

Either because it corresponds to many practical situations, or because
it is a possible way to model data with missing values, there is an
extensive literature on the question of sampling random processes.
 
Around the sixties, an important amount of publications in signal
processing was devoted to the reconstruction of the spectral density
of $(X_n)_{n\geq 1}$ from a sampled version $(X_{T_n})_{n\geq 1}$. In
case of deterministic sampling, this reconstruction is prevented by
the aliasing phenomenon, (which is easily understandable from formula
(\ref{ds}) below).  Several authors noticed that aliasing can be
suppressed by introducing some randomness in the sampling procedure.
See, without exhaustivity, \cite{Beu}, \cite{Mas0} and \cite{Silv}
where several random schemes are introduced. The idea of sampling
along a random walk was first proposed in \cite{Silv} where the
authors proved that under some convenient hypotheses, such a sampling
scheme is alias-free when the characteristic function of $S$, $\hat
S(\lambda)=\sum_{j\geq 1} S(j)e^{ij\lambda}$, is injective.

Later, in the domain of time series analysis, attention was
particularly paid to the effect of sampling on parametric families of
processes. For example the effect of deterministic sampling on ARMA or
ARIMA processes is studied in \cite{Bre}, \cite{Nie} and \cite{Str}
among others. The main result is that the ARMA structure is preserved,
the order of the autoregressive part being never increased after
sampling. More generally, the stability of the ARMA family by random
sampling along a random walk is proved in \cite{Kad} and \cite{Opp}.
Precisely, if $A(L)X_n=B(L)\varepsilon_n$ and $A_1(L)Y_n=B_1(L)\eta_n$
are the minimal representations of $\mathbf{X}$ and of the sampled
process $\mathbf{Y}$, the roots of the polynomial $A_1(z)$ belong to
the set $\{\sum_{j\geq 1}S(j)r_k^j\}$ where the $r_k$'s are the roots
of $A(z)$. As $S(1)\neq 1$ implies $\vert \sum_{j\geq
  1}S(j)r_k^j\vert< \vert r_k\vert$, a consequence is that the
convergence to zero of $\sigma_Y(h)$ is strictly faster than that of
$\sigma_X(h)$. So, random sampling an ARMA process along a random walk
shortens its memory. In \cite{Opp} it is even pointed out that some
ARMA processes could give rise to a white noise through a well chosen
sampling law.

Only few papers deal with the question of the memory of the
process obtained by sampling a long-memory one. The reader can find in  \cite{Cha} and
\cite{Hwa} a detailed study of deterministic sampling  and time
aggregation of the FARIMA ($0,d,0$) process with related statistic questions. These authors point out that deterministic sampling does not affect the value of the memory parameter $d$ of the process. In the present paper we deal with random sampling of long memory processes.

In all the sequel, a second order stationary
process  $\mathbf{X}$ is said to have long memory if its covariance sequence is non-summable 
\begin{equation}\label{long}
\sum_{h\geq 0}\vert\sigma_X(h)\vert=\infty.
\end{equation}

In section \ref{Gen} we present some related topics such that $L^p$-convergence of $\sigma_Y(h)$ and absolute continuity of the spectrum. We show in particular that short memory is always preserved as well as absolute continuity of the spectral measure. The main results of the paper,
concerning changes of memory by sampling processes with regularly varying covariances, are
gathered in Section \ref{Princ}. We show that the
intensity of memory of such processes is preserved if $\E(T_1)=\sum jS(j)<\infty$, while this
intensity decreases when $\E(T_1)=\infty$. For sufficiently heavy tailed $S$,
the sampled process has short memory, which is somehow not surprising since
with a heavy tailed sampling law, the sampling intervals can be quite
large. In section \ref{Sea} we consider processes
presenting long-memory with seasonal effects, and investigate the particular
effects of deterministic sampling. We show that in some cases the
seasonal effects can totally disappear after sampling.

\section{Some features unchanged by random sampling} \label{Gen}

Taking $p=1$ in Proposition \ref{propcov} below confirms an intuitive claim: random sampling cannot
produce long-memory from short memory. 
Propositions \ref{bounded}, \ref{espfinie} and \ref{espinf} state that, at least in all situations investigated in this paper, random sampling preserves the
existence of a spectral density.

\subsection{Preservation of summability of the covariance}

\begin{Prop}\label{propcov}~
\begin{enumerate}[(i)]
\item\label{item:3} Let $p\geq 1$. If  $\sum\vert \sigma_X\vert^p<\infty$, the same holds for $\sigma_Y$,
\item\label{item:4} In the particular case $p\in [1,2]$ both processes  
$\mathbf{X}$ and $\mathbf{Y}$ have spectral
densities linked by the relation
$$
f_Y(\lambda)=\frac{1}{2\pi}\sum_{-\infty}^{+\infty}\int_{-\pi}^\pi \left(
e^{-i\lambda}\hat S(\theta)\right)^j f_X(\theta) d\theta, \label{spectdens}
$$
where $\hat S(\theta)=\E(e^{i\theta T_1})$ is the characteristic function of $S$.
\end{enumerate}
\end{Prop}
\begin{proof}
 
The covariance sequence of the sampled process is given by 
\begin{equation}
  \begin{cases}
   \sigma_Y(0) = \sigma_X(0)\\
 \sigma_Y(h) = \E\left(\sigma_X(T_h)\right)=\sum_{j=h}^\infty \sigma_X(j)S^{*h}(j)\quad  h\geq 1\\
  \end{cases}
  \label{covech1}
 \end{equation}
where $S^{*h}$, the $h$-times convoluted of $S$ by itself, is the probability distribution of $T_h$.

 As the sequence $T_h$ is strictly increasing,
  $\sigma_X(T_h)$ is almost surely a subsequence of
  $\sigma_X(h)$. Then, (\ref{item:3}) follows from  
$$
  \sum_h\Big\vert\E\left(\sigma_X(T_h)\right)\Big\vert^p\leq
  \E\sum_h\vert\sigma_X(T_h)\vert^p\leq\sum_h
  \vert\sigma_X(h)\vert^p.
$$

The proof of (\ref{item:4}) is immediate, using (\ref{item:3})
\begin{eqnarray*}
f_Y(\lambda)&=&\frac{1}{2\pi}\sum_{j\in\Z}e^{-ij\lambda}\sigma_Y(j)=\frac{1}{2\pi}\sum_{j\in\Z}e^{-ij\lambda}\E(\sigma_X(T_j))\\
&=&\frac{1}{2\pi}\sum_{j\in\Z}\int_{-\pi}^\pi
e^{-ij\lambda}\E(e^{iT_j\theta})f_X(\theta)d\theta,
\end{eqnarray*}
where the series converges  in
$L^2([-\pi,\pi])$ when $p\neq 1$, the covariance being then square-summable without being summable.
 \end{proof}

\subsection{Preservation of the existence of a spectral density}

Concerning the existence of a spectral density some partial results are easily
obtained. 
Firstly, it is well known that the existence of a spectral density is preserved by deterministic sampling (see (\ref{ds}) below).
Second, from Proposition \ref{propcov} above it follows that, for any sampling
law,  the spectral density of $Y$ exists when the covariance of $X$ is square
summable. 
It should also be noticed that, when proving that the ARMA structure is preserved by random sampling, \cite{Opp} (see also \cite{Kad} for the multivariate case) gives an explicit form of the spectral density of $Y$ when $X$ is an ARMA process.

The three propositions below show that preservation of the existence of a spectral density by random sampling holds for all the models considered in the present paper. 

The proofs are based on the properties of Poisson kernel recalled in Appendix \ref{sec:few-techn-results}  
\begin{equation}\label{Poissdef}
P_s(t)=\frac{1}{2\pi}\left(\frac{1-s^2}{1-2s\cos t+ s^2}\right)\quad s\in[0,1[.
\end{equation}
and the representation given in Lemma \ref{preds} of the covariance of sampled process.
\begin{Lem}\label{preds}
For all $j\geq 0$,
\begin{eqnarray}
\sigma_Y(j)&=&\lim_{r\to 1^-}\int_{-\pi}^\pi e^{ij\theta} g(r,\theta)d\theta\label{et-apres} \\
\hbox{where}\nonumber\\ 
g(r,\theta)&=& {\frac{1}{4\pi}}\int_{-\pi}^\pi f_X(\lambda)\left(\frac{1}{1-re^{-i\theta}\hat
S(\lambda)}+\frac{1}{1-re^{-i\theta}\hat
S(-\lambda)}\right)d\lambda \label{g}\\
&=&\frac{1}{4}\int_{-\pi}^\pi f_X(\lambda)\left(\frac{1}{\pi}+P_{r\rho}(\tau-\theta)+P_{r\rho}(\tau+\theta)\right)d\lambda.\label{Pois}
\end{eqnarray}
\end{Lem}
\begin{proof}
  The proof of the lemma is relegated in Appendix
  \ref{sec:proof-lemma-refpreds}.
\end{proof}

\begin{Prop}\label{bounded}
If $f_X$ is bounded in a neighbourhood of zero, the sampled process $Y$ has a spectral density given by
\begin{equation}\label{dslim}
f_Y(\lambda)=\lim_{r\to 1}{\frac{1}{4\pi}}\int_{-\pi}^\pi f_X(\lambda)\left(\frac{1}{1-re^{-i\theta}\hat
S(\lambda)}+\frac{1}{1-re^{-i\theta}\hat
S(-\lambda)}\right)d\lambda
\end{equation}
\end{Prop}
\begin{proof}
In the sequel we write 
\begin{equation}\label{roto}
\hat S(\lambda)=\rho(\lambda)e^{i\tau(\lambda)},
\end{equation}
often denoted $\rho e^{i\tau}$ for the sake of shortness.

The proof of the proposition simply consists in exchanging the limit and integration in (\ref{et-apres}). 

Firstly, it is easily seen that, if $\theta\neq 0$, $g(r,\theta)$ has a limit as $r\to 1^-$. Hence the proof is complete provided that conditions of Lebesgue's theorem hold. 

As we can suppose that the sampling is not deterministic, 
$$
\vert \hat S(\lambda)\vert <1 \quad\forall \lambda\in]0,\pi]
$$
(see \cite{Fell}).
Hence, thanks to the continuity of $\vert \hat S(\lambda)\vert$, 
$$
\sup_{\vert\lambda\vert>\varepsilon}\vert \hat
S(\lambda)\vert<1\quad\forall \varepsilon>0.
$$
The integral (\ref{g}) is split in two parts: choosing $\varepsilon$ such that $f$ is bounded on
$I_\varepsilon=[-\varepsilon,\varepsilon]$ and using the fact that the integrand in (\ref{g}) is positive (see (\ref{pos})).

\begin{eqnarray*}
&&\int_{I_\varepsilon} f_X(\lambda)Re \left(\frac{1}{1-re^{-i\theta}\hat
S(\lambda)}+\frac{1}{1-re^{i\theta}\hat
S(\lambda)}\right)d\lambda\\
& \leq & (\sup_{I_\varepsilon} f_X) \int_{I_\varepsilon} Re
\left(\frac{1}{1-re^{-i\theta}\hat S(\lambda)}+\frac{1}{1-re^{i\theta}\hat S(\lambda)}\right)d\lambda
\end{eqnarray*}
which leads, thanks to Lemma \ref{int=1}, to
\begin{eqnarray}
\int_{I_\varepsilon} f_X(\lambda)Re \left(\frac{1}{1-re^{-i\theta}\hat
S(\lambda)}+\frac{1}{1-re^{i\theta}\hat
S(\lambda)}\right)d\lambda&\leq& 4\pi(\sup_{I_\varepsilon} f_X) g^*(r,\theta)\nonumber\\
&=&4\pi(\sup_{I_\varepsilon}f_X)\label{premier}.
\end{eqnarray}

Now, 
\begin{eqnarray*}
&&\int_{I^c_\varepsilon} f_X(\lambda)Re\left(\frac{1}{1-re^{-i\theta}\hat
S(\lambda)}+\frac{1}{1-re^{i\theta}\hat
S(\lambda)}\right)d\lambda\\
&=&\pi\int_{I^c_\varepsilon}f_X(\lambda)\left(\frac{1}{\pi}+P_{r\rho}(\tau-\theta)+P_{r\rho}(\tau+\theta)\right)d\lambda.\nonumber
\end{eqnarray*}
Applying (\ref{Poisson0})
with 
$$
s=r\rho(\lambda)<\rho(\lambda),\quad t=\tau(\lambda)\pm \theta, \quad \hbox{and}\quad 
\eta=1-\sup_{\vert\lambda\vert>\varepsilon}\vert \hat S(\lambda)\vert
$$ 
yields,
\begin{eqnarray}\label{second}
\pi\int_{I^c_\varepsilon}f_X(\lambda)&&\left(\frac{1}{\pi}+P_{r\rho}(\tau-\theta)+P_{r\rho}(\tau+\theta)\right)d\lambda \nonumber\\
\leq&&\left(1+\frac{2}{\eta}\right)\int_{I^c_\varepsilon}
f_X(\lambda)d\lambda\nonumber\\
\leq && \left(1+\frac{2}{\eta}\right)\int_{-\pi}^\pi
f_X(\lambda)d\lambda.
\end{eqnarray}
Gathering (\ref{premier})  and (\ref{second}) leads to the result via Lebesgue's theorem.
\end{proof}

In the next two propositions, the spectral density of $X$ is allowed to be unbounded at zero. Their proofs are in the Appendix. 

We first suppose that the sampling law has a finite expectation. It shall be proved in subsection \ref{preser} that in this case the intensity of memory of $X$ is preserved.

\begin{Prop}\label{espfinie}
If $\E T_1<\infty$ and if the spectral density of $X$ has the form

\begin{equation}
  f_X(\lambda)=\vert \lambda\vert ^{-2d}\phi(\lambda)\label{eq:lmspe}
\end{equation}
where $\phi$ is nonnegative, integrable and bounded in a neighbourhood of zero, and where $0\leq d <1/2$,

then the sampled process has a spectral density $f_Y$ given by 
(\ref{dslim}).
\end{Prop}
\begin{proof}
  See Appendix \ref{sec:proof-prop-refespf}
\end{proof}
The case $\E(T_1)=\infty$ is treated in the next proposition, under the extra assumption (\ref{reg}) meaning that $S(j)$ is regularly varying at infinity.  We shall see in subsection \ref{decrea} (see Proposition \ref{v3}) how the parameter $d$ is transformed when the sampling law satisfies condition (\ref{reg}). In particular we shall see that if $1<\gamma<3-4d$ the covariance of the sampled process is square summable, implying the absolute continuity of the spectral measure of $Y$. The proposition below shows that this property holds for every $\gamma\in ]1,2[$

\begin{Prop}\label{espinf}
Assume that the spectral density of $X$ has the form (\ref{eq:lmspe})
If the distribution $S$ satisfies the following condition 
\begin{equation}\label{reg}
S(j)\sim c j^{-\gamma} \;\hbox{when}\;j\to\infty
\end{equation}
where $1<\gamma <2$ and with $c>0$, 
then the conclusion of Proposition \ref{espfinie} is still valid.

\end{Prop}
\begin{proof}
  see Appendix \ref{sec:proof-prop-refesp}
\end{proof}

\section{Random sampling of processes with regularly varying covariances} \label{Princ}

 The propositions below are valid for the family of stationary processes whose covariance $\sigma_X(h)$ decays arithmetically, up to a slowly varying factor. In the first paragraph we show that if $T_1$ has a finite
 expectation, the rate of decay of the covariance is unchanged by sampling.
 We then see that, when
 $\E(T_1)=\infty$, the memory of the sampled process is reduced according to  the largest finite moment of $T_1$.

\subsection{Preservation of the memory when $\E(T_1)<\infty$ }\label{preser}

\begin{Prop}\label{fini} 
  Assume that the covariance of $X$ is of the form 
  $$
  \sigma_X(h)=
  h^{-\alpha}L(h), 
  $$
  where $0<\alpha<1$ and where $L$ is slowly
  varying at infinity and ultimately monotone (see \cite{Bin}).  If
  $\E(T_1)<\infty$, then, the covariance of the sampled process
  satisfies $$
  \sigma_Y(h) \sim  h^{-\alpha}(\E(T_1))^{-\alpha} L(h) $$
as $h\to \infty$.  
 \end{Prop}

\begin{proof}~
We  prove that, as $h\to\infty$
\begin{equation}
   \frac{\sigma_X(T_h)}{h^{-\alpha}L(h)}=
   \left(\frac{T_h}{h}\right)^{-\alpha}\frac{L(T_h)}{L(h)}\tops (\E(T_1))^{-\alpha}
\label{cvps}
 \end{equation} 
and  that, for $h$ large enough,
\begin{equation}\label{leb}
\left| \frac{\sigma_X(T_h)}{h^{-\alpha}L(h)}\right| \leq 1.
 \end{equation} 

Since $T_h$ is the sum of independent and identically distributed random variables $T_h=\sum_{j=1}^h
\Delta_j$, with common distribution $T_1\in L^1$,  the law of large numbers
leads to 
\begin{equation}\label{lgn}
\frac{T_h}{h}\tops \E(T_1).
\end{equation}
Now 
\begin{equation}\label{cvL}
\frac{L(T_h)}{L(h)}\tops 1,
\end{equation}
indeed, $T_h\geq h$ because the intervals $\Delta_j \geq
1$ for all $j$ and (\ref{lgn}) implies that for $h$ large enough
$$
T_h\leq  2\E(T_1) h,
$$
Therefore, using the fact that $L$ is ultimately monotone, we have  
\begin{equation}\label{double}
L(h)\leq L(T_h)\leq L(2\E(T_1) h)
\end{equation}
if $L$ is ultimately increasing (and the reversed inequalities if $L$ is
ultimately decreasing). Finally (\ref{double}) directly leads to (\ref{cvL}) since $L$ is
slowly varying at infinity (see Theorem 1.2.1 in \cite{Bin}). Clearly,
(\ref{lgn}) and (\ref{cvL}) imply the convergence (\ref{cvps}).

In order to prove (\ref{leb}), we write
$$
\frac{\sigma_X(T_h)}{h^{-\alpha}L(h)}=\left(\frac{T_h}{h}\right)^{-\alpha/2}\left(\frac{T_h}{h}\right)^{-\alpha/2}\frac{L(T_h)}{L(h)}=\left(\frac{T_h}{h}\right)^{-\alpha/2}
\frac{T_h^{-\alpha/2}L(T_h)}{h^{-\alpha/2}L(h)}.
$$
As $\alpha>0$ and $T_h\geq h$,
$$
\left(\frac{T_h}{h}\right)^{-\alpha}\leq 1.
$$
Moreover, $h^{-\alpha/2}L(h)$ is decreasing for $h$
large enough (see \cite{Bin}), so that
$$
\frac{T_h^{-\alpha/2}L(T_h)}{h^{-\alpha/2}L(h)}\leq 1.
$$
These two last inequalities  lead to (\ref{leb})

Since $\sigma_Y(h)= \E\left(\sigma_X(T_h)\right)$ we conclude the 
proof by applying  Lebesgue's theorem, and 
we get  
$$
\frac{\sigma_Y(h)}{h^{-\alpha}L(h)}\to (\E(T_1))^{-\alpha}. 
$$

\end{proof}
\subsection{Decrease of memory when $\E(T_1)=\infty$ }\label{decrea}

  If $\E(T_1)=\infty$ it is known that $T_h/h \to\infty$, implying that the limit in (\ref{cvps}) is zero.
  In other words, in this case the  convergence to zero of $\sigma_Y(h)$ could be faster than $h^{-\alpha}$.
  The aim of this section is to give a
  precise evaluation of this rate of convergence.

\begin{Prop}\label{v3} 
Assume that the covariance of $\mathbf{X}$ satisfies 
\begin{equation}\label{majX}
\vert\sigma_X(h)\vert \leq  c h^{-\alpha}  
\end{equation}
where $0<\alpha<1$.

If 
\begin{equation}
  \liminf_{x\to\infty} x^\beta P(T_1 >x) >0\label{eq:DR}
  \end{equation}
  for some $\beta\in (0,1)$
(implying $\E(T_1^\beta)= \infty$ ) then
\begin{equation}\label{maj}
\vert \sigma_Y(h)\vert \leq C h^{-\alpha/\beta}.
\end{equation}

\end{Prop}

\begin{proof}

From hypothesis (\ref{majX}), 
$$
\vert\sigma_Y(h)\vert\leq\E(\vert\sigma_X(T_h)\vert)\leq c\E(T_h^{-\alpha})
$$
Then,
\begin{eqnarray}
\E(T_h^{-\alpha})&=&\sum_{j=h}^{\infty}P(T_h\leq j)\left(j^{-\alpha}-(j+1)^{-\alpha}\right)\nonumber\\
&\leq& \alpha\sum_{j=0}^{\infty}P(T_h\leq j) j^{-\alpha-1}.\label{esp1}
\end{eqnarray}
From hypothesis (\ref{eq:DR}) on the tail of the sampling law, it follows that, for  large enough $h$,
\begin{eqnarray}
  P(T_h \leq j) &\leq&
  P\left(\max_{1\leq l \leq h} \Delta_l \leq j\right) =P\left(T_1\leq 
  j\right)^h \nonumber\\
  &\leq& \left(1-C j^{-\beta}\right)^h\leq e^{-\frac{Ch}{j^{\beta}}}. \label{esp2}
\end{eqnarray}
Gathering (\ref{esp1}) and (\ref{esp2}) then gives
\begin{eqnarray*}
\E(T_h^{-\alpha})\leq \alpha \sum_{j=h}^{\infty}j^{-\alpha-1}e^{-\frac{Ch}{j^{\beta}}}.
\end{eqnarray*}
The last sum has the same asymptotic behaviour as 
$$
\int_h^\infty x^{-\alpha-1}e^{-\frac{Ch}{x^{\beta}}}dx=h^{-\alpha/\beta}\int_0^{h^{1-\beta}}u^{\alpha/\beta-1}e^{-Cu} du, 
$$
and the result follows since, as $\beta <1$,

$$
\int_0^{h^{1-\beta}}u^{\alpha/\beta-1}e^{-Cu} du \xrightarrow[]{h\to\infty} \int_0^\infty u^{\alpha/\beta-1}e^{-Cu} du.
$$

\end{proof}
Next proposition states that the bound in Proposition \ref{v3} is sharp under some additional hypotheses.

\begin{Prop}\label{v2}
Assume that
$$
\sigma_X(h) =  h^{-\alpha} L(h)
$$
where $0<\alpha<1$ and where  $L$ is slowly varying at infinity and ultimately monotone

If
\begin{equation}\label{defbeta}
\beta =:\sup\set{ \gamma : \E(T_1^\gamma)<\infty } \in ]0,1[
\end{equation}
then, for every $\varepsilon> 0$,
\begin{equation}\label{minor1}
\sigma_Y(h) \geq C_1 h^{-\frac{\alpha}{\beta}-\varepsilon}. \end{equation}

\end{Prop}

\begin{proof}
Let $\epsilon > 0 $.  We have
\begin{equation*}
  \frac{\sigma_X(T_h)}{h^{-\frac{\alpha}{\beta}-\varepsilon}} =
  \frac{T_h ^{-\alpha}}{h^{-\frac{\alpha}{\beta}-\varepsilon}}  L(T_h) =  \frac{T_h^{-\alpha-\frac{\beta\epsilon}{2}}}{h^{-\frac{\alpha}{\beta}-\varepsilon}}  T_h^{\frac{\beta\epsilon}{2}} L(T_h)
=\left(\frac{T_h}{h^\delta}\right)^{-\alpha-\frac{\beta\epsilon}{2}}
T_h^{\frac{\beta\epsilon}{2}}  L(T_h)
\end{equation*}
where
$$
\delta=
\frac{\alpha/\beta+\varepsilon}{\alpha+\frac{\beta\epsilon}{2}}=\frac1{\beta}\left(\frac{\alpha+\beta\varepsilon}{\alpha+\beta\varepsilon/2}\right).
$$

Using Proposition 1.3.6 in \cite{Bin},
$$
T_h^{\frac{\beta\epsilon}{2}} L(T_h)
\xrightarrow[]{h\to\infty} +\infty \quad \hbox{a.s}
$$

Moreover  $\delta>\frac{1}{\beta}$. From (\ref{defbeta}), this implies  $\E(T_1^{1/\delta})<\infty$. Then, the law of large numbers of Marcinkiewicz-Zygmund (see \cite{stout} Theorem 3.2.3) yields
 \begin{equation}
   \frac{T_h}{h^\delta} \tops 0 \label{eq:MZ} \quad \hbox{as}\quad h\to \infty.    \end{equation}

Therefore by applying Fatou's Lemma
\begin{equation*}
\frac{\sigma_Y(h)}{h^{-\frac{\alpha}{\beta}-\varepsilon}} \xrightarrow[]{h\to\infty}
\infty.
\end{equation*}

\end{proof}

\subsection{The particular case of the FARIMA family}
A  FARIMA ($p,d,q$) process is defined from a white noise 
$(\varepsilon_n)_n$, a parameter $d\in ]0,1/2[$ and two polynomials $A(z)=z^p+a_1z^{p-1}+\ldots+a_p$ and $B(z)=z^q+b_1z^{q-1}+\ldots+b_q$ non vanishing on the domain $\vert z\vert\geq 1$, by
\begin{equation}\label{FARIMA}
X_n= B(L) A^{-1}(L)(I-L)^{-d}\varepsilon_n
\end{equation}
where $L$ is the back-shift operator $LU_n=U_{n-1}$.
The FARIMA ($0,d,0$)
$$
W_n=(I-L)^{-d}\varepsilon_n
$$
introduced by Granger \cite{Gra} is specially popular.

It is well known (see \cite{Bro}) that the covariance of the FARIMA ($p,d,q$) process satisfies 
\begin{equation}
\sigma_X(k)\sim c k^{2d-1}\quad \hbox{as} \;k\to\infty,
\end{equation}
allowing the above Propositions 
\ref{fini}, \ref{v3} and \ref{v2} to apply. The results can be
summarised as follows. 
\begin{Prop}\label{APhdsim}
Sampling a FARIMA process (\ref{FARIMA}), with a sampling law $S$ such that, with some $\gamma>1$
 \begin{equation}\label{heavy}
 S(j)=P(T_1 = j)\sim c j^{-\gamma}\quad \hbox{as} \;j\to\infty,
 \end{equation}
leads to a  process $(Y_n)_n$ whose 
 auto covariance function 
satisfies the following properties  
\begin{enumerate}[(i)]
\item\label{item:1}  if $\gamma>  2$,
\begin{equation}\label{less1}
    \sigma_Y(h)\sim  C h^{2d-1}.
\end{equation}
\item \label{item:2} if $\gamma \leq 2$,
\begin{equation}\label{less2}
    C_1 h^{(2d-1)/(\gamma-1) -\epsilon} \leq
    \sigma_Y(h)\leq C_2 h^{(2d-1)/(\gamma-1)},\quad \forall \epsilon>0.
\end{equation}
Consequently, the sampled process $(Y_n)_n$ has 
\begin{itemize}
\item   the same memory parameter  if  $\gamma \geq 2$, 
\item   a reduced long memory parameter if  $ 2(1-d)\leq \gamma < 2$,  
\item   short memory  if  $ 1<\gamma<2(1-d) $
\end{itemize}

\end{enumerate}

\end{Prop}
\begin{proof}~

For (\ref{item:1}), $T_1$ has a finite expectation. Hence  Proposition
\ref{fini} applies leading to
$$
\sigma_Y(h)\sim c h^{2d-1}.
$$ 

For (\ref{item:2}), the conditions of Propositions \ref{v3} and \ref{v2} in section \ref{decrea}
are satisfied with $\beta=\gamma-1$ and $\alpha=1-2d$, leading to (\ref{less2})
Since 
$$
\frac{1-2d}{\gamma-1}> 1-2d, \quad \hbox{if} \quad 1<\gamma\leq 2, 
$$
the intensity of memory is then reduced except to the case $\gamma=2$.

The loss of memory is such that if $1<\gamma<2(1-d)$
it happens that $\frac{1-2d}{\gamma-1}>1$, implying the convergence of the series
$\sum \vert\sigma_Y(h)\vert$. In this case random sampling has created
short-memory. 
\end{proof}

We illustrate this last result by simulating  and sampling FARIMA(0,$d$,0)
processes.

Simulations of the trajectories are based on the moving average
representation of the FARIMA (see \cite{bardet1}).

The sampling distribution is 
\begin{equation}\label{simul}
P(S= k) = \int_k^{k+1} (\gamma-1) t^{-\gamma} \d t  \sim C k^{-\gamma},
\end{equation}
which is simulated using the fact that when $u$ is uniformly distributed on
$[0,1]$, the integer part of $u^{1/(1-\gamma)}$ is distributed according (\ref{simul}).
 \begin{figure}[htbp]
\vskip 2.5cm
   \begin{center}
     \includegraphics[height=7cm, width=12cm]{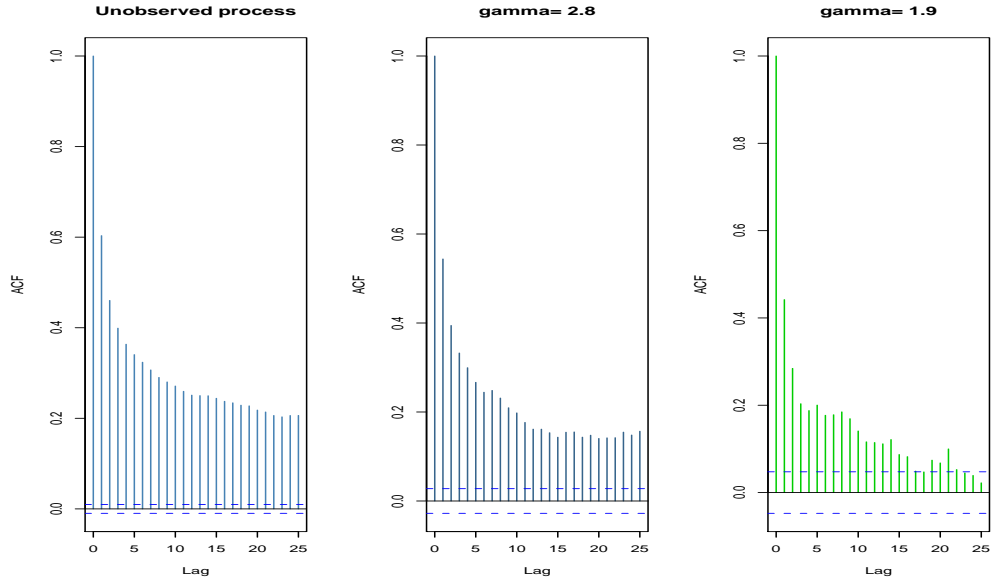}
   \end{center}
\vskip -2cm
  \caption{ Auto covariance functions of  $\mathbf{X}$  \textit{[left]}  and  
   of two sub sampled processes corresponding to  $\gamma= 2.8 $
   \textit{[middle]}  and $\gamma=1.9$ \textit{[right]}. The number of 
 observed values of the  sub sampled processes is equal to $5 000$  }
  \label{fig:acf}
\end{figure}

In Figure \ref{fig:acf} are depicted the empirical auto covariances of a trajectory of the
process $X$ and of two sampled trajectories $Y_1$ and $Y_2$ . The memory parameter of $X$ is
$d=0.35$ and the parameters of the two sampling distributions are respectively
$\gamma_1=2.8$ and $\gamma_2=1.9$.  According to Proposition \ref{APhdsim}, the
process $Y_1$ has the same memory parameter as $X$, while the memory intensity
is reduced for the process $Y_2$.

Then we estimate the memory parameters of the sampled processes by using
the FEXP procedure  introduced by \cite{SoulierMoulines2000,bardet2}. The FEXP
estimator is adapted to processes having a spectral density.
From Propositions \ref{espinf}
and \ref{espfinie}, this is the case for our sampled processes. 
Figure \ref{fig:dest} shows the estimate and a confidence interval. Two values, $d=0.1$  (lower
curves) and $d=0.35$ (upper curves) of the memory parameter of $X$ are
considered.  In abscissa, the parameter $\gamma$ of the sampling distribution is allowed to
vary between $1.7$ and $3.3$. From Proposition \ref{APhdsim} the 
value of the memory parameter is $d$ if $\gamma\geq 2$ and $d-1+\gamma/2$
otherwise. In the case $d=0.1$ short memory is obtained for $\gamma<1.8$. In
the case $d=0.35$, sample distributions leading to short memory are too heavy
tailed to allow tractable simulations. 

\begin{figure}[htbp]
  \begin{center}
    \includegraphics[height=7cm,width=14cm]{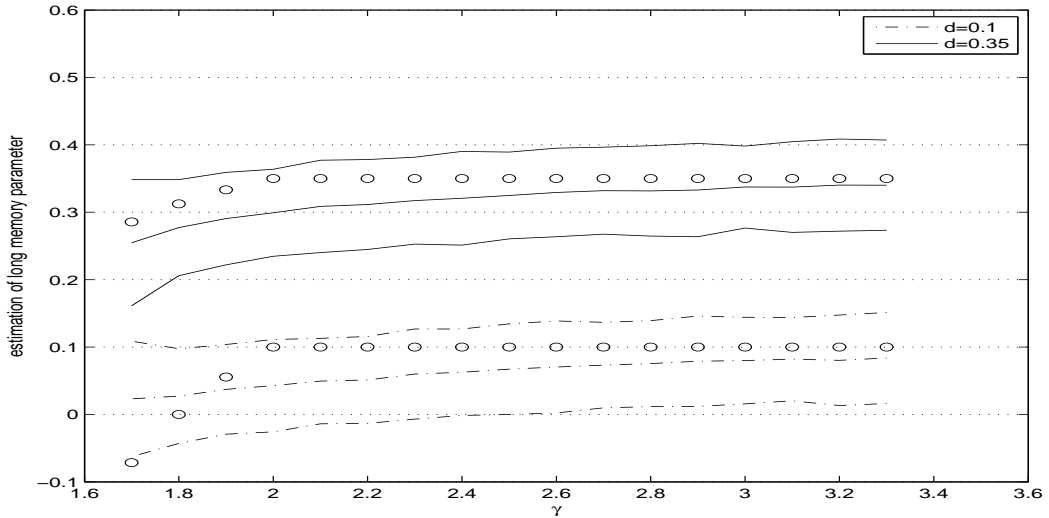}
  \end{center}
\caption{Estimation of the long memory parameter for $d=0.1$ and
    $d=0.35$ as function of parameter $\gamma$.  The confidence
    regions are obtained using 500 independent replications. The
    circles represent the theoretical values of the parameter obtained
    in Proposition \ref{APhdsim}. For each $\gamma$, the estimation of
    $d$ is evaluated on $5 000$ observations}
  \label{fig:dest}
\end{figure}

\section{Sampling  generalised fractional processes}\label{Sea}
 
Consider now the family of generalised fractional processes whose spectral
density has the form 
\begin{equation}\label{densseas}
f_X(\lambda)= \vert\Phi (\lambda)\vert^2\Big\vert \prod_{j=1}^M(e^{i\lambda}-e^{i\theta_j})^{-d_j}(e^{i\lambda}-e^{-i\theta_j})^{-d_j}\Big\vert^2,
\end{equation}
with exponents $d_j$ in $]0,1/2[$, and where $\vert\Phi (\lambda)\vert^2$ is
the spectral density of an ARMA process. 

\subsection{Decrease of memory by heavy-tailed sampling law}
See \cite{Gra}, \cite{Lei},
\cite{Oul} for results and references on this class of processes. 
Taking $M=1$ and $\theta_1=0$, it is
clear that this family includes the FARIMA processes, which belong to the
framework of section \ref{Princ}.
It is proved in \cite
{Lei} that  as $h\to \infty$ the covariance sequence is a sum of periodic
sequences damped by a regularly varying factor
\begin{equation}\label{covseas}
\sigma_X(h)=h^ {2d-1}\left(\sum c_{1,j}\cos(h\theta_j)+c_{2,j}\sin(h\theta_j) +o(1)\right)
\end{equation}
where $d=\max\{d_j; j=1,\ldots,M\}$ and where the sum extends over indexes
corresponding to $d_j=d$. 
Hence, these models generally
present seasonal long-memory, and for them the
regular form  $\sigma_X(h)=  h^{2d-1} L(h)$ is lost. Precisely, this
regularity  remains only when, in (\ref{densseas}),  $d=\max\{d_j; j=1,\ldots,M\}$ corresponds to the unique frequency $\theta=0$. In all other cases, from (\ref{covseas}), it has to be to be
replaced by 
$$
\sigma_X(h)=O(h^{2d-1}).
$$

From the previous comments it is clear that 
Propositions \ref{fini} and
\ref{v2} are no longer valid in the case of seasonal long-memory. That means
that in this situation, it is not sure that long memory is preserved by a sampling law having a
finite expectation. But it is true that long-memory can be lost. Indeed,
Proposition  \ref{v3} applies with $\alpha=1-2d$ where
$d=\max\{d_j; j=1,\ldots,M\}$, and  the following result holds:
\begin{Prop}
When sampling a process with spectral density (\ref{densseas}) along a random walk such that the sampling law satisfies (\ref{heavy}), the obtained process has short memory as soon as $1<\gamma<2(1-d)$ where $d=\max\{d_j; j=1,\ldots,M\}$.
\end{Prop}

\subsection{The effect of deterministic sampling}
\label{sec:dirac}

In a first step we just suppose that  $\mathbf{X}$ has a
spectral density $f_X$, which is  unbounded in one or several
frequencies (we call them singular frequencies). It is clearly the case of the
generalised fractional processes. The examples shall be taken in this family. 
We investigate the effect of deterministic sampling on the number and the values of the singularities.
Let us suppose that
$S=\delta_k$, the law concentrated at $k$.
It is well known that if  $\mathbf{X}$ has a spectral density, the spectral density of
the sampled process $\mathbf{Y}=(X_{k n})_n$ is 
\begin{eqnarray}\label{ds}
f_Y(\lambda)&=&\frac{1}{2 \ell+1}\sum_{j=-l}^l 
f_X\left(\frac{\lambda-2\pi j}{2 \ell+1}\right)\qquad\hbox{if}\; k=2 \ell+1\\
f_Y(\lambda)&=&\frac{1}{2 \ell}\left(\sum_{j=-\ell+1}^{\ell-1} f_X\left(\frac{\lambda-2
\pi j}{2\ell}\right)+
f_X\left(\frac{\lambda-2\pi \ell \hbox{sgn}(\lambda)}{2\ell}\right)\right)\qquad \hbox
{if}\; k=2\ell\nonumber
\end{eqnarray}
\begin{Prop}\label{Det}
Let the spectral density of $(X_n)_n$ have $N_X\geq 1$ singular frequencies. Then, denoting by $N_Y$ the number of singular frequencies of the spectral density of the sampled process $(Y_n=X_{kn})_n$
\begin{itemize}
\item $1\leq N_Y\leq N_X.$

\item $N_Y<N_X$ if and only if $f_X$ has at least two singular frequencies
  $\lambda_0\neq \lambda_1$ such that, for some integer number $j^*$,
$$
\lambda_0-\lambda_1=\frac{2\pi j^*}{k}
$$
\end{itemize}
\end{Prop}

\begin{proof}
  We choose $k=2\ell+1$ to simplify. The proofs for $k=2\ell$ are
  quite similar.  Now remark that for fixed $j$,
$$
\frac{\lambda-2\pi j}{2\ell+1}\in I_j:=\left[\frac{-\pi-2\pi
    j}{2\ell+1}, \frac{\pi-2\pi j}{2\ell+1}\right[,\quad
\hbox{for}\quad \lambda\in [-\pi,\pi[.
$$
The intervals $I_j$ are non overlapping and $\cup_{j=-\ell}^{\ell}
I_j=[-\pi,\pi[$.  Now suppose that $f_X$ is unbounded at some
$\lambda_0$. Then $f_Y$ is unbounded at the unique frequency
$\lambda^*$ satisfying $\lambda_0=\frac{\lambda^*-2\pi j_0}{2\ell+1}$
for a suitable value $j_0$ of $j$. In other words, to every singular
frequency of $f_X$ corresponds a unique singular frequency of $f_Y$.
As a consequence, the number of singular frequencies of $f_Y$ is at
least $1$, and cannot exceed the number of singularities of $f_X$.
  
It is clearly possible to reduce the number of singularities by
deterministic sampling. Indeed, two distinct singular frequencies of
$f_X$, $\lambda_0$ and $\lambda_1$ produce the same singular frequency
$\lambda^*$ of $f_Y$ if and only if
$$
\frac{\lambda^*-2\pi j_0}{2\ell+1}=\lambda_0,\quad
\hbox{and}\quad\frac{\lambda^*-2\pi j_1}{2\ell+1}=\lambda_1,
$$
which happens if and only if $\lambda_0-\lambda_1$ is a multiple of
the basic frequency $\frac{2\pi}{2\ell+1}$.

\end{proof}
Proposition \ref{Det} is illustrated by the three following examples showing
the various effects 
of the aliasing phenomenon when sampling generalised fractional processes.

We take $k=3$, so that the spectral density of $\mathbf{Y}$
is 
$$
f_Y(\lambda)=\frac{1}{3}\left(f_X\left(\frac{\lambda-2\pi}{3}\right)
  +f_X\left(\frac{\lambda}{3}\right)+f_X\left(\frac{\lambda+2\pi}{3}\right)\right).
$$

\begin{exx}
Firstly consider $f_X(\lambda)=\vert e^{i\lambda}-1\vert^{-2d}$, the
  spectral density of the FARIMA ($0,d,0$) process. Then
 $$
  f_Y(\lambda)=\frac{1}{3}\left(\vert
    e^{i\frac{\lambda-2\pi}{3}}-1\vert^{-2d}+\vert
    e^{i\frac{\lambda}{3}}-1\vert^{-2d}+\vert
    e^{i\frac{\lambda+2\pi}{3}}-1\vert^{-2d} \right) 
$$
  has, on
  $[-\pi,\pi[$, only one singularity at $\lambda=0$, associated with the same
  memory parameter  $d$ as for $f_X(\lambda)$.  The memory of $\mathbf{Y}$ has
  the same characteristics as
  the memory of $\mathbf{X}$. This was already noticed in \cite{Cha} and \cite{Hwa}.
\end{exx}
\begin{exx}
Consider now $f_X(\lambda)=\vert
  e^{i\lambda}-e^{\frac{2i\pi}{3}}\vert^{-2d}\vert
  e^{i\lambda}-e^{\frac{-2i\pi}{3}}\vert^{-2d}$, the spectral density of
  a long-memory seasonal process (see \cite{Oul}). The sampled spectral density is
 $$
  f_Y(\lambda)=\frac{1}{3}\sum_{j=-1}^1\vert
  e^{i\frac{\lambda-2\pi}{3}}-e^{\frac{2i\pi}{3}}\vert^{-2d}\vert
  e^{i\frac{\lambda-2\pi}{3}}-e^{\frac{-2i\pi}{3}}\vert^{-2d},
 $$
and is everywhere continuous on $[-\pi,\pi]$, except at $\lambda=0$ where
$$
f_Y(\lambda)\sim c\vert \lambda\vert^{-2d}.
$$
In other words the memory is preserved, but the seasonal effect
disappears after sampling. 
\end{exx}

\begin{exx}
Consider now a case of two seasonalities ($\pm \pi/4$ and $\pm 3\pi/4$)
  associated with two different memory parameters $d_1$ and $d_2$:
  $$
  f_X(\lambda)=\vert e^{i\lambda}-e^{\frac{i\pi}{4}}\vert^{-2d_1}\vert e^{i\lambda}-e^{\frac{-i\pi}{4}}\vert^{-2d_1}\vert e^{i\lambda}-e^{\frac{3i\pi}{4}}\vert^{-2d_2}\vert e^{i\lambda}-e^{\frac{-3i\pi}{4}}\vert^{-2d_2}.
 $$
It is easily checked that $f_Y$ has the same singular frequencies as $f_X$, with an exchange of the memory parameters: 

\begin{eqnarray*}
f_Y(\lambda)&\sim& c \Big |\lambda \pm\frac{\pi}{4}\Big |^{-2d_2} \qquad \hbox {near}\; \mp \pi/4\\
f_Y(\lambda)&\sim& c\Big |\lambda \pm\frac{3\pi}{4}\Big |^{-2d_1} \qquad \hbox
{near}\; \mp 3\pi/4.
\end{eqnarray*}
\end{exx}

\section{Appendix}
\subsection{Proof of Lemma \ref{preds}}
\label{sec:proof-lemma-refpreds}
Let us consider the two $z$-transforms of the bounded sequence $\sigma_Y(j)$:

\begin{eqnarray*}
\hat {\sigma}_Y^-(z)&=&\sum_{j=0}^\infty z^j\sigma_Y(j) \quad \vert z\vert<1
\end{eqnarray*}
and
\begin{eqnarray*}
\hat \sigma_Y^+(z)&=&\sum_{j=0}^\infty z^{-j}\sigma_Y(j)\quad \vert z\vert>1.
\end{eqnarray*}

On the first hand, from the representation 
$$
\sigma_Y(j)=\E(\sigma_X(T_j))=\int_{-\pi}^\pi f_X(\lambda)\hat S(\lambda)^jd_\lambda,
$$
we have 
\begin{eqnarray}
\hat {\sigma}_Y^-(z)&=&\sum_{j=0}^\infty z^{j}\int_{-\pi}^\pi f(\lambda)\left(\hat S(\lambda)\right)^j d\lambda
=
\int_{-\pi}^\pi \frac{f(\lambda)}{1-z \hat S(\lambda)}d\lambda,\quad \vert z\vert<1 \;\label{transf4}
\end{eqnarray}
\begin{eqnarray}
\hat \sigma_Y^+(z)&=& \sum_{j=0}^\infty z^{-j}\int_{-\pi}^\pi f(\lambda)\left( S(\lambda)\right)^j d\lambda
=
\int_{-\pi}^\pi \frac{f(\lambda)}{1-\frac{\hat S(\lambda)}{z}}d\lambda,\quad
\vert z\vert>1\;\label{transf4+}
\end{eqnarray}

On the second hand, let $C_r$ be the circle $\vert z\vert=r$. If $0<r<1$, for all $j\geq 0$ 

\begin{eqnarray}
{\frac{1}{2i\pi}}\int_{C_r}\hat {\sigma}_Y^-(z) z^{-j-1}
dz &=&{\frac{1}{2\pi}}\int_{-\pi}^\pi(r e^{i\theta})^{-j}\hat
{\sigma}_Y^-(r e^{i\theta})d\theta\nonumber\\
&=& \sum_{l=0}^\infty
 {\frac{r^{l-j}\sigma_Y(l)}{2\pi}}\int_{-\pi}^\pi e^{i(l-j)\theta} d\theta = \sigma_Y(j),
\label{transf1}
\end{eqnarray}

and, similarly, if $r>1$

\begin{eqnarray}
{\frac{1}{2i\pi}}\int_{C_r}\hat {\sigma}_Y^+(z) z^{j-1}
dz
&=&{\frac{1}{2\pi}}\int_{-\pi}^\pi(r e^{i\theta})^{j}\hat
{\sigma}_Y^+(r e^{i\theta})d\theta\nonumber\\
&=&\sum_{l=0}^\infty
 {\frac{r^{j-l}\sigma_Y(l)}{2\pi}}\int_{-\pi}^\pi e^{i(j-l)\theta} d\theta = \sigma_Y(j),
\label{transf1+}
\end{eqnarray}

Gathering (\ref{transf4}) with (\ref{transf1}) and (\ref{transf4+}) with (\ref{transf1+}) leads to
\begin{equation}
\sigma_Y(j)= 
\begin{cases}\label{+-}
 {\frac{1}{2\pi}}\int_{-\pi}^\pi
e^{-ij\theta}\left(r^{-j}\int_{-\pi}^\pi\frac{f(\lambda)}{1-re^{i\theta}\hat
    S(\lambda)}d\lambda\right)d\theta, & \text{if $r<1$} \\
{\frac{1}{2\pi}}\int_{-\pi}^\pi
e^{ij\theta}\left(r^{j}\int_{-\pi}^\pi\frac{f(\lambda)}{1-\frac{e^{-i\theta}\hat S(\lambda)}{r}}d\lambda\right)d\theta,& \text{if $r>1$}\\
\end{cases} 
\end{equation}

Changing the integrand in (\ref{+-}) into its conjugate when $r<1$, and $r$ for $1/r$ when $r>1$  leads to gives
\begin{equation}\label{repg}
\sigma_Y(j)=r^{-j}\int_{-\pi}^\pi
e^{ij\theta}g(r,\theta)d\theta, , \qquad \forall r\in[0,1[
\end{equation}
where  $g(r,\theta)$ is defined in (\ref{g}). As the first member does not depend on $r$, (\ref{et-apres}) is proved.

Let us now prove (\ref{Pois}).
Firstly,
\begin{multline*}
Im\left(\frac{1}{1-re^{-i\theta}\hat
S(\lambda)}+\frac{1}{1-re^{-i\theta}\hat
S(-\lambda)}\right)\\
=\frac{1}{2}\left(\frac{1}{1-re^{-i\theta}\hat
S(\lambda)}+\frac{1}{1-re^{-i\theta}\hat
S(-\lambda)}-\frac{1}{1-re^{i\theta}\hat
S(-\lambda)}-\frac{1}{1-re^{i\theta}\hat
S(\lambda)}\right)
\end{multline*}
is an odd function of $\lambda$. 
Hence, the imaginary part of the integrand in
(\ref{g}) disappears after integration. 
 
Secondly,
\begin{multline}\label{pos}
Re\left(\frac{1}{1-re^{-i\theta}\hat
S(\lambda)}+\frac{1}{1-re^{-i\theta}\hat
S(-\lambda)}\right)\\
=
\frac{1}{2}\left(\frac{1-r^2\vert \hat S(\lambda)\vert^2}{\vert 1-re^{-i\theta}\hat
S(\lambda)\vert^2}+\frac{1-r^2\vert \hat S(\lambda)\vert^2}{\vert 1-re^{-i\theta}\hat
S(-\lambda)\vert^2}+2\right),
\end{multline}
and the proof is over.

\subsection{A few technical results}
\label{sec:few-techn-results}
Hereafter we recall some properties used in the paper.

\begin{Lem}\label{Poisson}
The Poisson kernel  (\ref{Poissdef}) satisfies 

\begin{equation}\label{Poisson0}
2\pi P_s(t)\leq
\frac{1-s^2}{(1-s)^2}=\frac{1+s}{1-s}\leq \frac{2}{\eta}\qquad\forall s<1-\eta,\quad\hbox{with}\;s<1.
\end{equation}
 
\begin{equation}\label{Poisson1}
0<\delta<\vert t\vert\leq \frac{\pi}{2}\Longrightarrow  P_s(t)<P_s(\delta). 
\end{equation}

\begin{equation}\label{Poisson2}
2\pi \sup_{0<s<1}P_s(t)=\frac{1}{\vert \sin t\vert}\qquad \forall t\in ]-\pi/2,\pi/2[.
\end{equation}
\end{Lem}

\begin{proof}

The bound (\ref{Poisson0}) is direct. Inequality (\ref{Poisson1}) comes from the monotonicity of the kernel with respect to $s$.

Let us prove (\ref{Poisson2}): 
$$
\frac{\partial}{\partial s}\left(P_s(t)\right)=2 \frac{s^2\cos t-2s+\cos t}{(1+s^2-2s\cos t)^2}.
$$
It is easily checked that the numerator has one single root $s_0=\frac{1-\vert \sin
  t\vert}{\cos t}$ in $[0,1[$, and 
  $
  P_{s_0}(t)=\vert \sin t\vert^{-1}.
  $ 
\end{proof}
The following result comes from spectral considerations.
\begin{Lem}\label{int=1}
With $\rho=\rho(\lambda)$ and $\tau=\tau(\lambda)$ defined in (\ref{roto}),
for all $r\in [0,1[$ and $\theta\in ]-\pi,\pi[$, 
\begin{equation}
g^*(r,\theta):=\frac{1}{4}\int_{-\pi}^\pi \left(\frac{1}{\pi}+P_{r\rho}(\tau-\theta)+P_{r\rho}(\tau+\theta)\right)d\lambda=1
\end{equation}
\end{Lem}
\begin{proof}
Notice first that if $f_X \equiv 1$, the process $X$ is a white noise with $\var (X_1)=2\pi$.  Hence, the sampled process is also a white noise with variance $2\pi$. Applying  (\ref{g}) to a white noise,  we get 
\begin{equation}
  2\pi \delta_0(j)= r^{-j}\int_{-\pi}^\pi e^{ij\theta} g^*(r,\theta) d\theta.\label{eq:star}
  \end{equation}
for all $r\in[0,1[$. Since $r^{j}\delta_0(j)= \delta_0(j) $, we can rewrite (\ref{eq:star}) as 
$$
2\pi \delta_0(j)= \int_{-\pi}^\pi e^{ij\theta} g^*(r,\theta) d\theta.
$$
This means that $g^*(r,\theta)$ is the Fourier transform of $(2\pi
\delta_0(j))_j$. Consequently $g^*(r,\theta)\equiv 1$.
\end{proof}
\subsection{Proof of Proposition \ref{espfinie}} \label{sec:proof-prop-refespf}

As for Proposition \ref{bounded}, the  proof consists in finding an integrable function $g(\theta)$ such that
\begin{equation}\label{dominée}
\vert g(r,\theta)\vert \leq g(\theta)\quad \forall r\in]0,1[,\theta\in]-\pi,\pi[.
\end{equation}

For that purpose, we need the following estimation of $\hat S(\lambda)$ near zero.
\begin{eqnarray}
\vert 1-\hat S(\lambda)\vert&=&\Big \vert(1-e^{i\lambda})\sum_{j\geq 1}S(j) \frac{1-e^{ij\lambda}}{1-e^{i\lambda}}\Big \vert =\vert 1-e^{i\lambda}\vert
\Big\vert \sum_{j\geq 1}S(j) e^{\frac{(j-1)\lambda}{2}}\frac{\sin (j\lambda/2)}{\sin\lambda/2}\Big\vert\nonumber\\
&\leq& \vert 1-e^{i\lambda}\vert \sum_{j\geq 1}jS(j)\nonumber\\
&=&\vert 2i\sin(\lambda/2)e^{i\lambda/2}\vert \sum_{j\geq 1}jS(j)=2\vert \sin(\lambda/2)\vert \sum_{j\geq 1}jS(j)\leq \vert\lambda\vert \sum_{j\geq 1}jS(j)\nonumber\\
&=&C\vert\lambda\vert\label{caror}.
\end{eqnarray}

Now we use the fact that
$$
\vert u\vert \leq u_0<1\Longrightarrow \vert 1-u\vert>1-u_0\quad\hbox{and}\quad\sin(\arg(1-u))< u_0.
$$
Since $u_0< \pi/2$, this implies $\arg(1-u)<\pi u_0/2$.

From this and inequality (\ref{caror}) we obtain
\begin{equation}\label{arg}
\vert \lambda\vert\leq \frac{1}{C}\Longrightarrow \vert \tau(\lambda)\vert
\leq \frac{\pi C\vert\lambda\vert}{2}
\end{equation}

For a fixed $\theta >0$ let $\lambda_0$ be such that $\phi$ is bounded on $[-\lambda_0,\lambda_0]$. Denoting
$$
b(\theta)=\min\left\{\lambda_0,\frac{1}{C},\frac{\theta}{\pi C}\right\},
$$
we separate 
$[-\pi,\pi]$ into four intervals:
$$
[-\pi,-b(\theta)[,\;[-b(\theta),0[,\;]0,b(\theta)],\; ]b(\theta),\pi].
$$
In the sequel we only deal with the two last intervals, and concerning the integrand in (\ref{Pois}) we only treat the part $P_{r\rho}(\tau-\theta)$:
\begin{eqnarray*}
I_1(\theta)+I_2(\theta)&=&\int_0^{b(\theta)}f(\lambda)P_{r\rho(\lambda)}(\tau(\lambda)-\theta)d\lambda+\\
&+&\int_{b(\theta)}^{\pi}f(\lambda)P_{r\rho(\lambda)}(\tau(\lambda)-\theta)d\lambda.
\end{eqnarray*}

$\bullet$ Bounding $I_1$:

From (\ref{arg}), since $b(\theta)\leq \theta/(\pi C)$, we have 
$
\vert \tau(\lambda)\vert\leq \theta/2
$
which implies
$$
\frac{\theta}{2} \leq  \vert \tau(\lambda)-\theta\vert.
$$ 
Via (\ref{Poisson1}) and (\ref{Poisson2}), this leads to
$$
P_{r\rho}(\tau-\theta)\leq P_{r\rho}(\theta/2)\leq \frac{1}{2\theta}.
$$
Consequently
\begin{eqnarray*}
I_1&\leq& \frac{ \sup_{[0,b(\theta)]} (\phi(\lambda))}{2\theta}
\int_0^{b(\theta)}\lambda^{-2d}d\lambda= \frac{\sup_{[0,b(\theta)]}
  (\phi(\lambda))}{2\theta} (b(\theta))^{-2d+1}\\
&\leq& C_1\theta^{-2d}
\end{eqnarray*}
since $b(\theta)\leq \theta/(\pi C)$ and $-2d+1>0$.\\

$\bullet$ Bounding $I_2$:

When $\lambda>b(\theta)$, we have
$\lambda^{-2d}\leq C_2\max\{\theta^{-2d},1\}$ for some constant $C_2$. Hence
\begin{equation}\label{I2}
I_2\leq C_2\max\{\theta^{-2d},1\}\int_{-\pi}^{\pi}\phi(\lambda)P_{r\rho}(\theta-\tau)d\lambda.
\end{equation}
Since
$\phi$ is bounded in a neighbourhood of zero, the arguments used to prove  Proposition \ref{bounded} show that the integral in (\ref{I2}) is bounded by a constant. 

Finally $I_1+I_2$ is bounded by an integrable function $g(\theta)$ and the proposition is proved.

\subsection{Proof of Proposition \ref{espinf}}\label{sec:proof-prop-refesp}
The following lemma gives the local behaviour of $\hat
S(\lambda)$ under assumption (\ref{reg}).
 
\begin{Lem}\label{control}
Since $S(j)\sim c j^{-\gamma}$ with $1<\gamma<2$ and $c>0$,
\begin{equation}\label{mod}
\vert\lambda\vert^{1-\gamma}(1-\hat S(\lambda))\to Z,\;\hbox{if}\;\lambda\to 0.
\end{equation}
where $Re(Z)>0$ and $Im(Z)<0$.
\end{Lem}
\begin{proof}
From the assumption on $S(j)$, 
\begin{eqnarray*}
\sum_{j\geq 1}S(j)(1-\cos(j\lambda))&\sim_{\lambda\to 0}&c \sum_{j\geq
  1}j^{-\gamma}(1-\cos(j\lambda))\\
\end{eqnarray*}
and
\begin{eqnarray*}
\sum_{j\geq 1}S(j) \sin(j\lambda)&\sim_{\lambda\to 0}&c \sum_{j\geq
  1}j^{-\gamma}\sin(j\lambda).
\end{eqnarray*}
Then,  using well known results of Zygmund (\cite {Zyg}; pages 186 and 189)
\begin{eqnarray*}
\sum_{j\geq
  1}j^{-\gamma}(1-\cos(j\lambda))&\sim_{\lambda\to 0}& \frac{c}{\gamma-1} \vert
\lambda\vert^{\gamma -1}=: c_1\vert
\lambda\vert^{\gamma -1}.\\
\end{eqnarray*}
and 
\begin{eqnarray*}
\sum_{j\geq
  1}j^{-\gamma}\sin(j\lambda)&\sim_{\lambda\to 0}& c \Gamma(1-\gamma)\cos\left(\frac{\pi\gamma}{2}\right)\vert
\lambda\vert^{\gamma -1}=:c_2\vert
\lambda\vert^{\gamma -1}\\
\end{eqnarray*}
It is clear that $c_1>0$, and $c_2<0$ follows from the fact that $\Gamma(x)<0$ for $x\in ]-1,0[$ and  
$\cos\left(\pi\gamma/2\right)<0$.  
\end{proof}


In the sequel we take $\lambda>0$. From this lemma, if $\lambda$ is small enough (say
$0\leq \lambda \leq \lambda_0$),
\begin{eqnarray}
c_3
\lambda ^{\gamma-1}\leq&\tau(\lambda)&\leq c'_3
\lambda^{\gamma-1}\label{taulong}
\end{eqnarray}
and
\begin{eqnarray}
1-c_4 
\lambda ^{\gamma-1}\leq &\rho(\lambda)&\leq  1-c'_4
\lambda^{\gamma-1}\label{rholong}
\end{eqnarray}
where the constants are positive.

For a fixed $\theta>0$, we deduce from (\ref{taulong})
\begin{equation}\label{liminfarg}
\lambda<\min\left\{\lambda_0,\left(\frac{\theta}{c'_3}\right)^{1/(\gamma-1)}\right\}
\quad\hbox{implies}\quad0<\theta-c'_3\lambda^{\gamma-1}\leq \theta-\tau(\lambda)
\end{equation}

After choosing $\lambda_1\leq\lambda_0$ such that $\phi$ is bounded on $[0, \lambda_1]$ define
$$
c(\theta)=\min\left\{\lambda_1,\left(\frac{\theta}{c'_3}\right)^{1/(\gamma-1)}\right\}.
$$
Then we split $[-\pi,\pi]$ into six intervals
$$
[-\pi,-\lambda_1[,\;[-\lambda_1,
-c(\theta)/2[,\;[-c(\theta)/2,0[, \;]0, c(\theta)/2],\; ]c(\theta)/2, \lambda_1],\;]\lambda_1,\pi].
$$
 
We only consider the integral on the three last domains and  
the part $P_{r\rho}(\tau-\theta)$ of the integrand in (\ref{Pois}).
 
$\bullet$ When $\lambda\in ]0, c(\theta)/2]$,  inequality (\ref{liminfarg}) and properties (\ref{Poisson1}) and 
(\ref{Poisson2}) of the Poisson kernel lead to
$$  
P_{r\rho}(\theta-\tau)\leq P_{r\rho}(\theta-c'_3\lambda^{\gamma-1})\leq \frac{C}{\theta-c'_3\lambda^{\gamma-1}},
$$
whence 
\begin{eqnarray*}
I_1&=&\int_0^{c(\theta)/2} f(\lambda)P_{r\rho}(\theta-\tau)d\lambda
\leq
C'_1\int_0^{c(\theta)/2}\frac{\lambda^{-2d}}{\theta-c'_3\lambda^{\gamma-1}}d\lambda\\
&\leq& C'_1\int_0^{\frac{1}{2}\left(\frac{\theta}{c'3}\right)^{1/(\gamma-1)}}\frac{\lambda^{-2d}}{\theta-c'_3\lambda^{\gamma-1}}d\lambda\\
&=& C'_2\theta^{\frac{-2d+1}{\gamma-1}-1}\int_0^{2^{-\frac{1}{\gamma-1}}}\frac{u^{\frac{-2d+1}{\gamma-1}-1}}{1-u}du.
\end{eqnarray*}
Since $\frac{-2d+1}{\gamma-1}-1>-1$, the last integral is finite, implying 
$$
I_1\leq C'_3 \theta^{\frac{-2d+1}{\gamma-1}-1},
$$
which is an integrable function of $\theta$.

$\bullet$ Thanks to (\ref{Poisson0}) and to the r.h.s. of (\ref{rholong}), we
have on the interval $]c(\theta)/2, \lambda_1]$
$$
P_{r\rho}(\theta-\tau)\leq \frac{C'_3}{\lambda^{\gamma-1}}. 
$$
Since  $\phi$ is bounded on this domain,
\begin{eqnarray*}
I_2&\leq& C'_3 \sup_{0<\lambda\leq\lambda_1}{\phi(\lambda)} \int_{c(\theta)/2}^{\lambda_1}\lambda^{-2d+1-\gamma}d\lambda\\
&=& \frac{C'_3 \sup_{0<\lambda\leq\lambda_1}{\phi(\lambda)}}{-2d+2-\gamma}\left( \lambda_1^{-2d+2-\gamma}-(c(\theta)/2)^{-2d+2-\gamma}\right)\\
&\leq&\frac{C'_3 \sup_{0<\lambda\leq\lambda_1}{\phi(\lambda)}}{-2d+2-\gamma}\left( \lambda_1^{-2d+2-\gamma}+(c(\theta)/2)^{-2d+2-\gamma}\right)\\
&=&  C'_4\left(\lambda_1^{-2d+2-\gamma}+\left(\min\left\{\lambda_1,\left(\frac{\theta}{c'3}\right)^{1/(\gamma-1)}\right\}\right)^{-2d+2-\gamma}\right).
\end{eqnarray*}
where the function between brackets is integrable because 
$$
\frac{-2d+2-\gamma}{\gamma-1}=\frac{-2d+1}{\gamma-1}-1>-1.
$$

$\bullet$ Finally, 
\begin{eqnarray*}
I_3=\int_{\lambda_1}^\pi f(\lambda)P_{r\rho}(\theta-\lambda)d\lambda
\leq \lambda_1^{-2d}\int_{-\pi}^\pi \phi(\lambda)P_{r\rho}(\theta-\lambda)d\lambda
\end{eqnarray*}
which has already been treated since $\phi$ is bounded near zero.\\

Gathering the above results on $I_1$, $I_2$ and  $I_3$ completes the proof.

\end{document}